\documentclass[11pt,a4paper]{amsart}
\usepackage[margin=1in]{geometry}
\usepackage{amssymb,amsmath,exscale,latexsym,amsthm,graphics,graphics,mathtools,overpic,subcaption}\usepackage{epsfig}    
\usepackage{epic}      
\usepackage{eepic}     
\usepackage[matrix,arrow,curve,cmtip]{xy} 
\usepackage{amsfonts}
\usepackage{enumerate} 
\usepackage{paralist}
\usepackage{color}
\usepackage{mathrsfs}
\usepackage{enumitem}
\usepackage{MnSymbol}  
\usepackage{hyperref}

\theoremstyle{plain}
        \newtheorem{theorem}{Theorem}
        \newtheorem*{theorem*}{Theorem}
        \newtheorem*{maintheorem*}{Main Theorem}
        \newtheorem*{conj*}{Conjecture}
        \newtheorem{lemma}[theorem]{Lemma}
        \newtheorem{cor}[theorem]{Corollary}
        \newtheorem{prop}[theorem]{Proposition}

\theoremstyle{definition}
        
        \newtheorem*{definition*}{Definition}

\theoremstyle{remark}
        \newtheorem*{remark}{Remark}

        \newtheorem*{example}{Example}
        
        \newtheorem*{claim1}{Claim 1}
        \newtheorem*{claim2}{Claim 2}

\numberwithin{equation}{section}

\def\reminder #1 {{\sf #1}}
\def\hide #1 {}
\long\def\longhide #1 {}

\newcommand\numberthis{\addtocounter{equation}{1}\tag{\theequation}}

\makeatletter
\newcommand{\mylabel}[2]{#2\def\@currentlabel{#2}\label{#1}}
\makeatother

\newcommand{\R}{\mathbb{R}}      
\newcommand{\C}{\mathbb{C}}      
\newcommand{\N}{\mathbb{N}}      
\newcommand{\Z}{\mathbb{Z}}      
\newcommand{\CDach}{\widehat{\mathbb{C}}}
\newcommand{\D}{\mathbb{D}}      

\newcommand{\GC}{\mathcal{G}}

\newcommand{\sG}{\mathsf{G}}
\newcommand{\sT}{\mathsf{T}}

\renewcommand{\:}{\colon}

\newcommand{\Sp}{S^2}

\newcommand{\fix}{\operatorname{Fix}}

\newcommand{\inter}{\operatorname{int}}
\newcommand*\interior[1]{#1^{\mathsf{o}}}

\newcommand{\Rat}{\operatorname{Rat}}

\newcommand{\Tisch}{\operatorname{Tisch}}
\newcommand{\Charge}{\operatorname{Charge}}

\begin{document}

\title{Tischler graphs of critically fixed rational maps and their applications}

\author{Mikhail Hlushchanka}
\address{Korteweg-de Vries Instituut voor Wiskunde, Universiteit van Amsterdam,  1090 GE \newline Amsterdam, The Netherlands}
\address{Mathematisch Instituut,  Universiteit Utrecht,
3508 TA Utrecht,  The Netherlands}
\email{mikhail.hlushchanka@gmail.com}

\date{June 15, 2023}

\keywords{Critically fixed rational maps,  combinatorial classification problem,  Tischler graphs,  blow-up surgery, curve attractor.}
\subjclass[2010]{Primary 37F20, 37F10} 

\begin{abstract}
A rational map $f\colon\CDach\to\CDach$ on the Riemann sphere $\CDach$ is called critically fixed if each critical point of $f$ is fixed under $f$. In this article, we study the properties of a combinatorial invariant, called the Tischler graph, associated with such a map.   We show that the Tischler graph of a critically fixed rational map is always connected, establishing a conjecture made by Kevin Pilgrim.  
This result allows us to solve two classical open problems in rational dynamics in the setting of critically fixed rational maps, namely the combinatorial classification problem and the global curve attractor problem.    
In particular,  we prove that there is a canonical one-to-one correspondence between the conjugacy classes of critically fixed rational maps and the isomorphism classes of connected planar embedded graphs. 
\end{abstract}
\maketitle

\section{Introduction}\label{sec:intro}

Fix an integer $d\geq 2$, and let $\Rat_d[\C]$ be the space of all rational maps of degree $d$ with complex coefficients. Each function $f\in \Rat_d[\C]$ may be viewed as a self-map $f\colon\CDach\to\CDach$ of the Riemann sphere $\CDach$. For $n\in\N$, we define $f^n=f \circ f \circ \cdots \circ f$ to be the $n$-fold composition of $f$ with itself, called the \emph{$n$-th iterate of $f$}. The iterates of $f$ yield a holomorphic dynamical system on~$\CDach$. 
One of the main challenges in the field is to describe (and distinguish) dynamically different maps within a particular family of rational maps in combinatorial terms, with the ultimate goal of better understanding the structure of the parameter space $\Rat_d[\C]$. 

For a rational map $f\colon \CDach\to \CDach$, we denote by $C_f$ the set of all critical points of $f$, that is, points $c\in\CDach$ at which $f$ is not locally injective. By the Riemann--Hurwitz formula, $f$ has $2\deg(f)-2$ critical points when counted with multiplicity.  Over 100 years ago, Fatou and Julia established that the global dynamical behavior of $f$ is controlled by the 
orbits of its critical points; see, for instance, \cite{Milnor_Book}.  We denote by $P_f:=\bigcup_{n=1}^{\infty} f^n(C_f)$ the \emph{postcritical set} of $f$. The rational map $f$ is said to be \emph{postcritically-finite} if the cardinality $\#P_f$ is finite,  in other words,
if every critical point of $f$ has a finite orbit under iteration. 
The set of fixed points of $f$ is denoted by $\fix(f)$. From the holomorphic fixed point formula \cite[Lemma 12.1]{Milnor_Book}, $\#\fix(f)=\deg(f)+1$ if counted with multiplicity.

In this article, we study properties of critically fixed rational maps.

\begin{definition*}
A rational map $f\colon  \CDach\to\CDach$ is said to be \emph{critically fixed} if $C_f\subset \fix(f)$,  that is,  if every critical point of $f$ is also a fixed point of $f$.
\end{definition*}

The class of critically fixed rational maps is very special: for every $d\geq 2$ there are only finitely many critically fixed rational maps of degree $d$ up to conformal conjugation (this is a consequence of Thurston's rigidity theorem \cite{DH_Th_char};  see, for instance, \cite[Corollary 3.7]{Census}).  Also, they are obviously postcritically-finite. At the same time, exceptional properties of critically fixed rational maps make it possible to elegantly answer many open dynamical questions for them. For instance, Tischler provided a complete combinatorial classification of critically fixed polynomials in terms of alternating planar embedded trees in  \cite[Theorem 4.2]{Tischler}. Cordwell et al.\ attempted to extend Tischler's considerations to the case of general critically fixed rational maps in \cite{Pilgrim_crit_fixed}. 
However, their work was not complete. In this article, we provide the main missing ingredient, which addresses the connectivity properties of a certain graph naturally associated with every critically fixed rational map.  We also finalize the combinatorial classification and discuss its application to the global curve attractor problem.   
Before giving more details about the main results, we introduce some necessary definitions.

\subsection{Tischler graphs}
Let  $f\colon\CDach\to\CDach$ be a critically fixed rational map and $c\in C_f$ be a critical point of $f$. The \emph{basin of attraction of $c$} is the set $$B_c := \{z \in \CDach : \lim_{n\to\infty} f^n(z) = c\}.$$ 
The connected component of $B_c$ containing the point $c$ is called the \emph{immediate basin} of $c$ and denoted by $\Omega_c$. 
 It follows from 
 \cite[Theorem 9.3]{Milnor_Book} that $\Omega_c$ is a simply connected open set. Moreover, there exists a conformal map $\tau_{c}\colon \D\to\Omega_c$ and a number $d_c\in\N$ such that
\[( \tau_{c} \circ f \circ \tau_c^{-1})(z) = z^{d_c}\]
for all $z$ in the open unit disc $\D:=\{z\in\C: |z|<1\}$. The number $d_c$ equals the local degree of $f$ at~$c$, that is, $d_c -1$ is the multiplicity of the critical point $c$.  Furthermore, the conformal map $\tau_{c}$ extends to a continuous and surjective map $\tau_{c}\colon \overline{\D}\to\overline{\Omega}_c$ between the closures; see \cite[Theorem 17.14 and Lemma 19.3]{Milnor_Book}.

An \emph{internal ray of angle $\theta\in\R$} in the immediate basin $\Omega_c$ is the image of the radial arc $r(\theta) := \{t e^{2\pi i \theta} \: t\in [0,1]\}$ under the 
map $\tau_c$. The point $\tau_c(e^{2\pi i \theta})\in\partial\Omega_c$ is called the \emph{landing point} of the internal ray of angle $\theta$. Note that the internal ray of angle $\theta$ is fixed (set-wise) under $f$ if and only if $\theta\equiv \frac{j}{d_c-1} \, \,(\text{mod } \Z)$ for some $j=0,\dots,d_c-2$.

\begin{definition*}
The \emph{Tischler graph} of a critically fixed rational map $f$ is the planar embedded graph $\Tisch(f)$ whose edge set consists of the fixed internal rays in the immediate basins of all critical points of $f$ and whose vertex set consists of the endpoints of these rays. That is, as a subset of $\CDach$,  $\Tisch(f)$ is the union of all fixed internal rays constructed in the previous paragraph. 
\end{definition*}

Now we are ready to formulate one of the main results of this article.

\begin{theorem}\label{thm:Tisch-connected}
Let $f\colon\CDach\to\CDach$ be a critically fixed rational map with $\deg(f)\geq 2$. Then the Tischler graph $\Tisch(f)$ is connected.
\end{theorem}

The statement above, conjectured by Kevin Pilgrim, is crucial for the combinatorial classification of critically fixed rational maps and has other important applications that we discuss below.

\subsection{Combinatorial classification problem} One of the main open problems in holomorphic dynamics asks to find a complete combinatorial description of all postcritically-finite rational maps in terms of finite data. More precisely,  given a postcritically-finite rational map, one first wants to assign a certain finite certificate, or \emph{model}, to it. Then, one wants to describe all the models that arise and determine whether there is a one-to-one correspondence between a class of postcritically-finite rational maps (up to conformal conjugacy) and a catalog of models, that is, provide a \emph{classification} of maps from the class.  Solving this problem is a major step towards understanding the structure of parameter spaces of rational maps in general. 

First combinatorial models, given by finite invariant graphs, were constructed for polynomial maps by Douady and Hubbard in the 1980's \cite{DH_Orsay}.  Later these models were used to classify all postcritically-finite polynomials \cite{BFH_Class,Poirier}.  However, the case of non-polynomial rational maps remains wide open and draws a lot of attention.  To date there are very few classification results and they are typically quite involved.  Besides polynomials,  we refer the reader to \cite{RussellDierk_Class,FixedNewton,CubicNewton} for various classification results for \emph{Newton maps} (these are rational maps that arise when applying Newton's root-finding method to polynomials).

Completing the work started in \cite{Pilgrim_crit_fixed}, we prove the following result.

\begin{theorem}\label{thm:classif_cr_fixed}
There is a canonical bijection between the conformal conjugacy classes of critically fixed rational maps of degree at least two and the isomorphism classes of connected planar embedded graphs with at least one edge and no loops.
\end{theorem}

We note that planar embedded graphs considered in this paper are allowed to have multiple edges.  The bijection in Theorem \ref{thm:classif_cr_fixed} is constructed using the ``blowing up'' operation from \cite[\S 2.5]{PT}; see Section \ref{sec:classification} for more details.

\subsection{Global curve attractor problem} 
Let $f\colon\CDach\to\CDach$ be a postcritically-finite rational map, and let $\mathscr{C}(f)$ be the set of all simple closed curves in $\CDach\setminus P_f$.  We denote by $[\gamma]$ the isotopy class of a curve $\gamma\in \mathscr{C}(f)$ relative to $P_f$, and by $\mathcal{C}(f)$ the set of all such isotopy classes.  A \emph{pullback} of  a curve $\gamma\in\mathscr{C}(f)$ under $f$ is a connected component of $f^{-1}(\gamma)$.  By the isotopy lifting property (see, for example, \cite[Lemma 6.9 and Proposition 11.3]{THEBook}), the map $f$ defines a \emph{pullback relation} $\xleftarrow{f}$ on the elements of~$\mathcal{C}(f)$: we set $[\gamma] \xleftarrow{f} [\delta]$ whenever $\delta$ is a pullback of $\gamma$ under $f$.

The global curve attractor problem asks the following question: Given a postcritically-finite rational map $f\colon \CDach\to\CDach$, is there a finite set $\mathcal{A}(f)\subset \mathcal{C}(f)$ such that, for every curve $\gamma\in\mathscr{C}(f)$,  all pullbacks $\delta$ of $\gamma$ under $f^n$ satisfy $[\delta]\in \mathcal{A}(f)$ for all sufficiently large $n\in\N$? The minimal set $\mathcal{A}(f)$ with this property (if it exists) is called the \emph{global curve attractor} of $f$.

It is conjectured that a finite global curve attractor exists for all postcritically-finite rational maps with \emph{hyperbolic orbifold}; see, for example, 
 \cite[Conjecture 1.1]{Pilgrim_Pullback}. We note that postcritically-finite rational maps with non-hyperbolic (i.e., parabolic) orbifold are very special and are well-understood; see the discussion in \cite{DH_Th_char} and \cite[Chapter 3]{THEBook}. In particular,  every critically fixed rational map $f$ with $\#C_f\geq 3$ has a hyperbolic orbifold.

In this article, we provide a positive answer to the global curve attractor problem for critically fixed rational maps. 

\begin{theorem}\label{thm:finite-curve-attractor}
Each critically fixed rational map has a finite global curve attractor. 
\end{theorem}

To the best of our knowledge,  critically fixed rational maps were the first family of rational maps of arbitrary degree for which the global curve attractor problem was solved.  Recently,  the problem was also solved for all postcritically-finite  polynomials \cite{LifitingTrees}.  Some other explicit examples of rational maps with exactly four postcritical points are studied in \cite{Lodge_Boundary, Lodge_Quadratic,BHI_Obstructions}.  However, for general postcritcally-finite rational maps the conjecture above remains wide open.

We also note that the mapping properties of simple closed curves play an important role in \emph{Thurston's characterization of rational maps} \cite{DH_Th_char}. In particular,  the pullback relation defined above is closely related to the  \emph{Thurston pullback map} from the proof of Thurston's theorem; see \cite{Lodge_Boundary} for details.

\subsection{Further applications and connections} Below,  we briefly discuss some further applications of Tischler graphs and other recent related work.

Since each edge of the Tischler graph $\Tisch(f)$ of a critically fixed rational map $f$ is invariant under $f$, an arbitrary spanning tree in $\Tisch(f)$ is an $f$-invariant planar embedded tree. Consequently, we have the following immediate corollary. 

\begin{cor}\label{cor:crit-fix-1-tile-subd}
Let $f$ be a critically fixed rational map. Then there exists a finite planar embedded tree with the vertex set containing $C_f$ that is invariant under $f$. 
\end{cor}

The invariant tree as above defines a \emph{one-tile subdivision rule} for the dynamics of $f$ in the sense of Cannon--Floyd--Parry \cite{CFP_Subdivision}. Furthermore, it appears to be very useful for the computation of the \emph{iterated monodromy group} of the map $f$ as discussed in \cite[Chapter 5]{H_Thesis}. 

Iterated monodromy groups were introduced by Nekrashevych in 2001 as (self-similar) groups that are naturally associated with certain dynamical systems, such as the iteration of a post\-critically-finite rational map $f\colon \CDach\to\CDach$ \cite{NekraIMG02}.  The iterated monodromy group contains all the essential information about the dynamics of $f$: one can reconstruct from the iterated monodromy group the action of the map on its Julia set.  

An invariant tree as in Corollary \ref{cor:crit-fix-1-tile-subd} drastically simplifies the computations of the iterated monodromy group action for critically fixed rational maps carried in \cite{Pilgrim_crit_fixed}.  Furthermore, it makes it possible to show that the iterated monodromy groups of critically fixed rational maps have quite ``exotic''  properties from the point of view of classical group theory. More precisely, the following result is proven in \cite[Corollary 5.6.6 and Corollary 6.4.3]{H_Thesis}.

\begin{theorem}\label{thm:img-crit-fix}
The iterated monodromy group of a critically fixed rational map $f$ with ${\#C_f\geq 3}$ is an amenable group of exponential growth.
\end{theorem}

In this way, critically fixed rational maps form a large class of maps for which we have a good understanding of the properties of iterated monodromy groups. 

Theorem \ref{thm:classif_cr_fixed} provides a base ingredient for the combinatorial classification of the \emph{topological analogs} of critically fixed rational maps, that is, orientation-preserving branched covering maps of the topological $2$-sphere for which every critical point is fixed; see \cite{HluPro}.  In the same paper,  combinatorial models of critically fixed rational maps were used to study the \emph{twisting problem},  which, loosely speaking, asks how the combinatorial model of a given rational map changes if we postcompose
it by a homeomorphism.

Finally,  we would like to mention similar classification results in the setting of \emph{critically fixed anti-rational maps}, that is,  the complex conjugates $\bar f$ of rational maps $f\colon \CDach\to \CDach$ for which all points in $C_{\bar f}$ are fixed \cite{LLM_Class, Geyer}.  Such maps also have intriguing connections to Kleinian groups \cite{LLM_Class} and to the theory of gravitational lensing \cite{GravLens}.

\subsection*{Structure of the article} First, we review some graph theoretical notions that we use in this article in Section \ref{sec:graph-theory}. Then, we provide an example of a critically fixed rational map and its Tischler graph in Section \ref{sec:example}. The connectivity of Tischler graphs, that is, Theorem~\ref{thm:Tisch-connected}, is proven in Section~\ref{sec:proof-main}. We discuss the combinatorial classification of critically fixed rational maps and prove Theorem \ref{thm:classif_cr_fixed} in Section~\ref{sec:classification}. Finally, we establish Theorem~\ref{thm:finite-curve-attractor}, which answers the global curve attractor problem for critically fixed rational maps,  in Section \ref{sec:global-curve-attractor}.

\subsection*{Notation} The cardinality of a set $S$ is denoted by $\#S$. We write $\N$, $\Z$, $\R$, and $\C$ for the sets of positive integers, integers, real numbers, and complex numbers, respectively. 
The Riemann sphere is denoted by $\CDach:=\C \cup \{\infty\}$. 

For a subset $U\subset \CDach$, we denote by $\overline{U}$, $\interior{U}$, and $\partial U$ the topological closure, the interior, and the boundary of $U$ in $\CDach$, respectively. For a Jordan arc $e$ in $\CDach$ joining two distinct points $z_1$ and~$z_2$, we set $\inter(e):= e\setminus\{z_1,z_2\}$.

Let $f\colon \CDach\to \CDach$ be a branched covering map, for instance,  a rational map. We denote by $\deg(f)$ the topological degree of $f$.  For $n\in\N$ we denote by $f^n$ the $n$-th iterate of $f$.  
The sets of  fixed points, critical points, and postcritical points of $f$ are denoted by $\fix(f)$, $C_f$, and $P_f$, respectively.  
If $U\subset \CDach$, then $f^{-n}(U)$ stands for the preimage of $U$ under $f^n$, that is, $f^{-n}(U)=\{z\in \CDach : f^n(z)\in U\}$. For simplicity, we set $f^{-n}(z) \coloneq f^{-n}(\{z\})$ for $z\in \CDach$.  Finally, we denote the restriction of $f$ to $U$ by~$f|U$. 

\section{Planar embedded graphs}\label{sec:graph-theory}

The goal of this section is to define planar embedded graphs, introduce related
constructions, and set up the notation. 

A \emph{closed arc} $e$ in $\CDach$ is the image $\eta(\mathopen[0{,}1\mathclose])$ of the closed unit interval $\mathopen[0{,}1\mathclose]$ under a continuous map $\eta\colon \mathopen[0{,}1\mathclose]\to \CDach$ such that $\eta|\mathopen(0{,}1\mathclose)$ is injective. The images $\eta(0)$ and $\eta(1)$ are called the \emph{endpoints} of the arc~$e$, which \emph{joins} them.  If $\eta(0) = \eta(1)$, we call $e$ a \emph{loop}, otherwise, $e$ is a \emph{Jordan arc}. The set $\inter(e):=\eta\left(\mathopen(0{,}1\mathclose)\right)=e\setminus\{\eta(0), \eta(1)\}$ is called the \emph{interior} of $e$. Note that $\inter(e)$ is not the same as $\interior{e}$. 

Formally, a \emph{planar embedded graph} $\sG$ is a pair $(V,E)$, where $V$ is a non-empty finite set of points in the Riemann sphere $\CDach$ and $E$ is a finite set of Jordan arcs in $\CDach$, such that
\begin{enumerate}[label=\text{(\roman*)},font=\normalfont]
\item for each $e\in E$, both endpoints of $e$ are in $V$;
\item for all distinct $e,e'\in E$, their interiors $\inter(e)$ and $\inter(e')$ are disjoint.
\end{enumerate}
The sets $V$ and $E$ are called the \emph{vertex} and \emph{edge} sets of $\sG$, respectively. Note that our notion of a planar embedded graph does not allow loops, but it does allow \emph{multiple edges}, that is, distinct edges that join the same pair of vertices. 

Let $\sG=(V,E)$ and $\sG'=(V',E')$ be two planar embedded graphs. We say that $\sG$ is (\emph{planar}) \emph{isomorphic}
to $\sG'$ if there exists an orientation-preserving homeomorphism $\phi\colon \CDach\to\CDach$ such that
it bijectively maps the vertices and edges of $\sG$ to the vertices and edges of $\sG'$, that is,  $V'= \phi(V)$ and $E'=\{\phi(e): e\in E\}$. Clearly,  this notion induces an equivalence relation on the set of all planar embedded graphs. An equivalence class of this relation is called an \emph{isomorphism class
of planar embedded graphs}.

In the following, we assume that $\sG=(V,E)$ is a planar embedded graph. 
The \emph{degree} of a vertex $v$ in $\sG$, denoted $\deg_\sG(v)$, is the number of edges of $\sG$ incident to $v$. Note that $2\#E = \sum_{v\in V} \deg_\sG(v)$.

The graph $\sG$ is called \emph{bipartite} if the vertices of $\sG$ can be partitioned into two disjoint sets $V_1$ and $V_2$ so that every edge of $\sG$ joins a vertex in $V_1$ to a vertex in $V_2$. In this case, the vertex subsets $V_1$ and $V_2$ are called the \emph{parts} of the graph $\sG$.

A \emph{walk between vertices $v$ and $v'$} in $\sG$ is a sequence $v_0=v,e_0,v_1,e_1,\dots,e_{n-1},v_n=v'$, where $e_j$ is an edge of $\sG$ that joins the vertices $v_{j}$ and $v_{j+1}$ for each $j=0,\dots,n-1$. A walk $v_0,e_0,$ $v_1,e_1,\dots,e_{n-1},v_n$ with $v_0=v_n$ is called a \emph{circuit of length $n$} in $\sG$ and is denoted by $(e_0,e_1,\dots,e_{n-1})$. 

A \emph{subgraph} of $\sG$ is a planar embedded graph $\sG'=(V',E')$ with $V'\subset V$ and $E'\subset E$.
 A \emph{connected component} of $\sG$ is a subgraph $\sG'=(V',E')$ such that 
\begin{enumerate}[label=\text{(\roman*)},font=\normalfont]
\item for all $v',v''\in V'$, there exists a walk between $v'$ and $v''$ in $\sG$; 
\item for all $v'\in V'$ and $v\in V\setminus V'$, there is no walk between $v$ and $v'$ in $\sG$;
\item $E'$ consists of all edges of $\sG$ with both endpoints in $V'$.
\end{enumerate}
The graph $\sG$ is called \emph{connected} if it has a unique connected component.

The subset $\GC:=V\cup \bigcup_{e\in E} e$ of $\CDach$ is called the \emph{realization} of $\sG$. A \emph{face} of $\sG$ is a connected component of $\CDach\setminus\GC$. 

As follows from \cite[Lemma 4.2.2]{DiestelGraph}, the topological boundary $\partial U$ of each face $U$ of $\sG$ may be viewed as the realization of a subgraph of $\sG$. Furthermore, a walk around a connected component of the boundary $\partial U$ traces a circuit $(e_0,e_1,\dots,e_{n-1})$ in $\sG$ such that each edge of $\sG$ appears zero, one, or two times in the sequence $e_0,e_1,\dots,e_{n-1}$. We will say that the circuit $(e_0,e_1,\dots,e_{n-1})$ \emph{bounds} the face $U$ or \emph{traces} (a connected component of) the boundary $\partial U$. 

\begin{example}
Consider the planar embedded graph $\sG$ shown in Figure \ref{fig:Graph_G}. It has four vertices, denoted by $v_1, \, v_{-1},\, v_c, \, v_{-c}$, and four edges, denoted by $e_1, \,  e_2,\, e_3, \, e_4$. Clearly, $\sG$ is connected. It has two faces, denoted by $U_1$ and $U_2$, such that $U_1$ is bounded by the circuit $(e_1,e_3,e_2,e_2,e_4,e_1)$ and $U_2$ is bounded by the circuit $(e_3,e_4)$.
\end{example}

\begin{figure}
  \centering
  \begin{subfigure}[b]{0.45\textwidth}
    \begin{overpic}
      [width=7cm, tics=10,
      ]{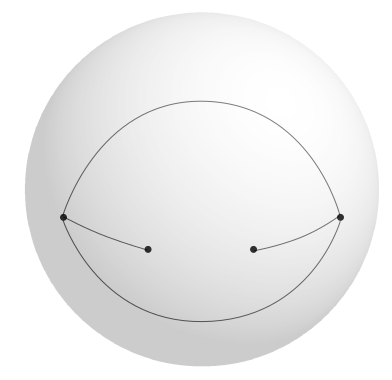}
      \put(50,10){$e_4$}
      \put(50,73){$e_3$}
      \put(26,38){$e_1$}
      \put(74,38){$e_2$}
      \put(50,35){$U_1$}
      \put(18,70){$U_2$}
      \put(87.5,37){$v_1$}
      \put(9.5,37){$v_{-1}$}       
      \put(39.5,32){$v_{-c}$}
      \put(59,32){$v_c$}   
    \end{overpic}
\caption{}
    \label{fig:Graph_G}
  \end{subfigure}
  \hfill
  \begin{subfigure}[b]{0.45\textwidth}
    \begin{overpic}
      [width=7cm, tics=10,
      ]{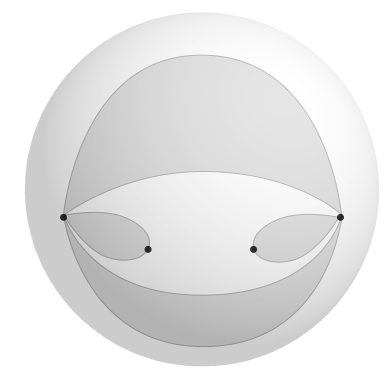}
      \put(87.5,37){$v_1$}
      \put(9.5,37){$v_{-1}$}        
      \put(35,16){$D_4$}
      \put(35,61){$D_3$}
      \put(23,36){$D_1$}
      \put(75,36){$D_2$}
      \put(87.5,37){$v_1$}
      \put(9.5,37){$v_{-1}$}         
      \put(50,35){$U'_1$}
      \put(18,70){$U'_2$}
      \put(50,85){$e'_3$}
      \put(50,55){$e''_3$}
      \put(50,23){$e'_4$}
      \put(50,10){$e''_4$}      
      \put(32,40){$e'_1$}
      \put(66,40){$e'_2$}
      \put(29,30){$e''_1$}
      \put(69,30){$e''_2$}            
    \end{overpic}
   \caption{}
    \label{fig:Graph_Gprime}
  \end{subfigure}
  \caption{A planar embedded graph $\sG$ (left) and its blow-up graph $\sG'$ (right).}
\label{fig:Graphs_Gs}
\end{figure}

\section{An example of a critically fixed rational map}\label{sec:example}

\begin{figure}[t]
  \centering
  \begin{subfigure}[b]{0.45\textwidth}
    \begin{overpic}
      [width=7cm, tics=10,
      ]{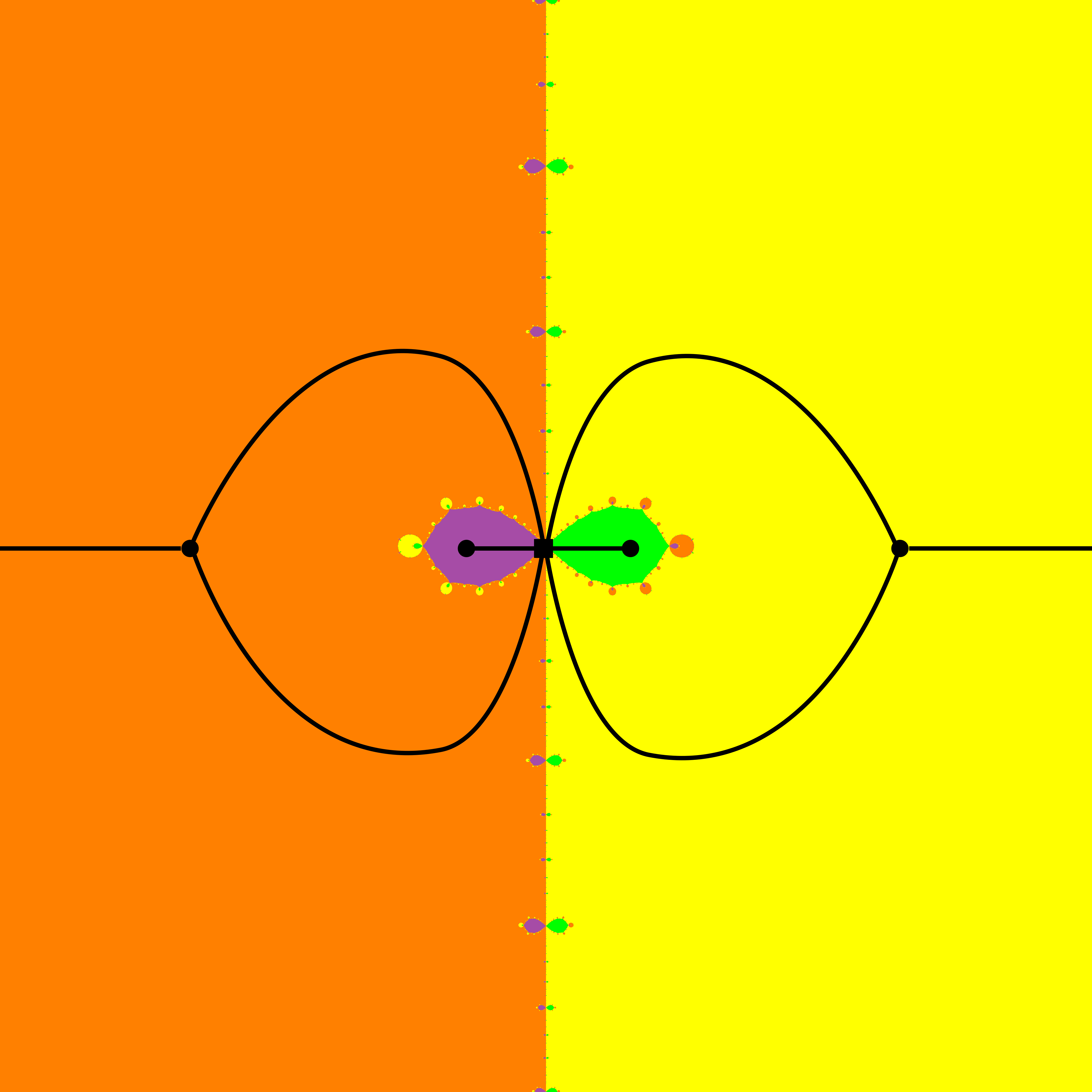}
      \put(83,45){$1$}
      \put(11.5,45){$-1$}     
      \put(40,50.5){$-c$}
      \put(55.5,50.5){$c$}   
       \put(48.75,41){$0$}     
      \put(28,48){$Q_1$}
      \put(68,48){$Q_2$}
      \put(48,72){$Q_3$}
      \put(48,22){$Q_4$}      
    \end{overpic}
\caption{}
    \label{fig:Tisch_f}
  \end{subfigure}
  \hfill
  \begin{subfigure}[b]{0.45\textwidth}
    \begin{overpic}
      [width=7cm, tics=10,
      ]{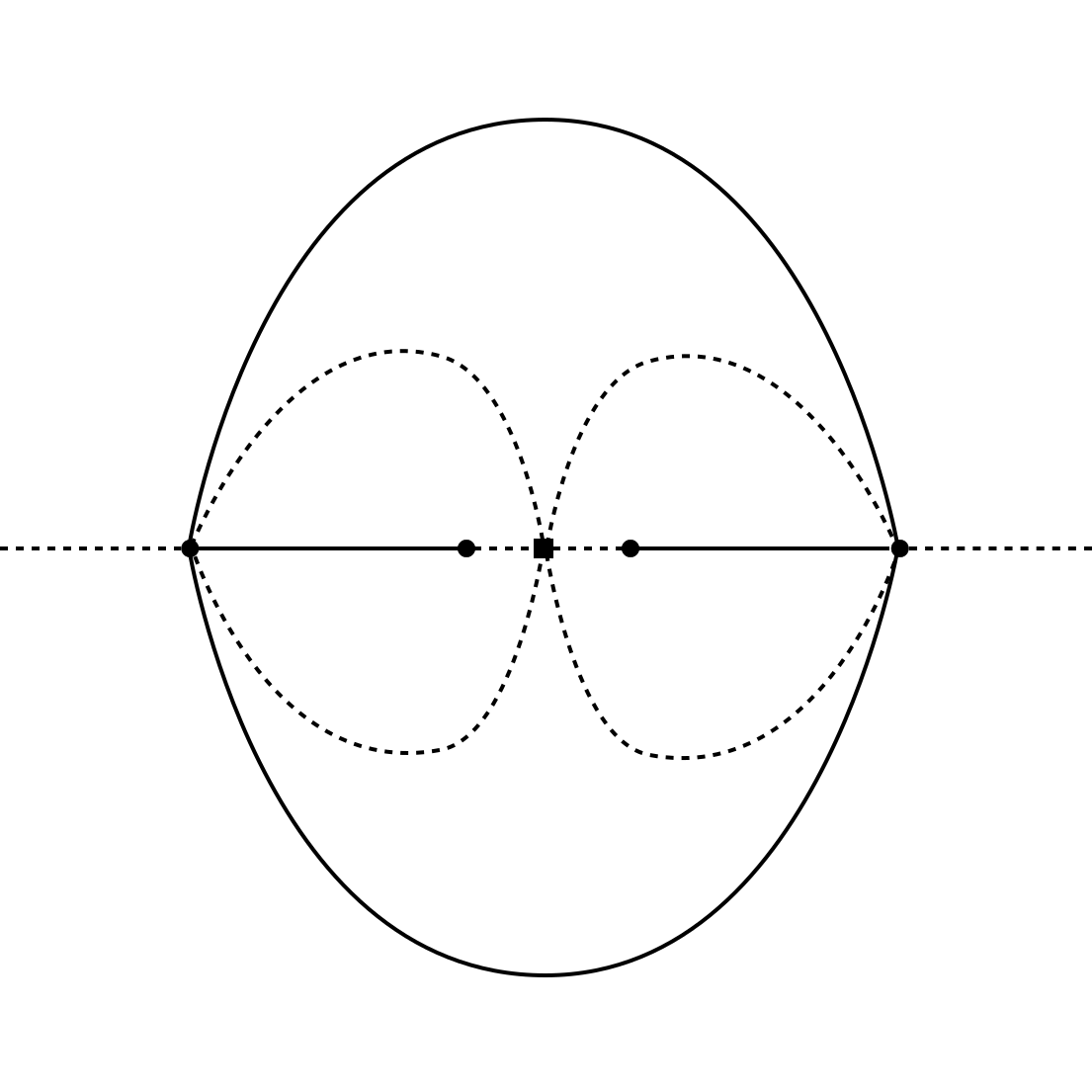}
      \put(83,45){$1$}
      \put(11.5,45){$-1$}     
      \put(40,46){$-c$}
      \put(55.5, 46){$c$}   
      \put(48.75,41){$0$}     
      \put(24,52){$e(Q_1)$}
      \put(62,52){$e(Q_2)$}
      \put(43,83){$e(Q_3)$}
      \put(43,13){$e(Q_4)$}         
    \end{overpic}
   \caption{}
    \label{fig:Charge_f}
  \end{subfigure}
  \caption{The Tischler graph $\Tisch(f)$ (left) and a charge graph $\Charge(f)$ (right) of the map $f$ from Section \ref{sec:example}.}
\label{fig:Tisch_and_Charge}
\end{figure}

In this section, we introduce a specific critically fixed rational map that we will use throughout the article to illustrate various constructions and phenomena.  Other explicit examples of critically fixed rational maps can be found in \cite[\S 11]{Pilgrim_crit_fixed}.

Consider the rational map $f\colon \CDach \to \CDach$ of degree $5$ given by
\begin{equation}\label{eq:f_G}
f(z)=\frac{z}{2} \,\frac{(5- \sqrt{5})z^4+ 10 (\sqrt{5} - 1) z^2 - 5 (7 - 3 \sqrt{5})}{5z^4 + 10 (\sqrt{5} - 2) z^2 +(2 \sqrt{5}-5)}, \quad z\in \CDach.
\end{equation}
The sets of critical and fixed points of $f$ are $$C_f=\{-1,\,-c,\,c, \,1\} \text{ and } \fix(f)=\{-1,\,-c,\,0,\,c,\, 1,\, \infty\},$$ respectively, where $c=\sqrt{5}-2$. 
In particular,  the map $f$ is critically fixed. The critical points $-1$ and $1$ have multiplicity~$3$, while the critical points $-c$ and $c$ have multiplicity~$1$.

The Tischler graph of $f$ is illustrated in Figure \ref{fig:Tisch_f} (note that $\Tisch(f)$ has a vertex at $\infty$ connected to the critical points $-1$ and $1$).  In the same figure,  the basins of attraction of the critical points $-1$, $-c$, $c$, and $1$ are shown in orange, purple, green, and yellow color, respectively. Note that $\Tisch(f)$ is connected and bipartite (with parts $\{-1,\,-c,\,c, \,1\}$ and $\{0,\infty\}$); its vertex set coincides with $\fix(f)$; and each of its faces, denoted by $Q_1$, $Q_2$, $Q_3$, $Q_4$, is bounded by a circuit of length $4$.  The faces $Q_1$ and $Q_2$ are bigons with a sticker inside, while the faces $Q_3$ and $Q_4$ are quadrilaterals.

\section{Connectivity of Tischler graphs}\label{sec:proof-main}

Let  $f\colon\CDach\to\CDach$ be a critically fixed rational map of degree $d\geq 2$. Suppose that $\sT:=\Tisch(f)$ is the Tischler graph of $f$ with the vertex set $V$, the edge set $E$, and the face set $F$.  In this section, we show that $\sT$ is a connected graph, that is, prove Theorem \ref{thm:Tisch-connected}.

It follows from the definition that $\sT$ is a bipartite graph with parts $C_f$ and $R_f$, where $R_f$ is the set of landing points of the fixed internal rays. Since each point $v\in R_f$ is fixed under $f$, we have $V=C_f\sqcup R_f\subset \fix(f)$ and 
\begin{equation}\label{eq:fix-Tisch}
\# V \leq \#\fix(f)=d+1.
\end{equation}
Note that the degree $\deg_\sT(c)$ of a critical point $c\in C_f$ in the graph $\sT$ coincides with the multiplicity of $c$. So, $\deg_\sT(c)= d_c -1$, where $d_c$ denotes the local degree of $f$ at the critical point $c$. Now, the Riemann-Hurwitz formula implies that
\begin{equation}\label{eq:RH-Tisch} 
2d-2 = \sum_{c\in C_f} (d_c-1) = \sum_{c\in C_f} \deg_\sT(c) =   \# E,
\end{equation}
where the last equality follows from the fact that $\sT$ is a bipartite graph with $C_f$ being one of its parts.

\begin{claim1}
If a pair of distinct edges $e_1,e_2\in E$ forms a circuit of length $2$ in $\sT$ (that is, $e_1\cup e_2$ is a Jordan curve in $\CDach$), then each of the two complementary components of  $e_1\cup e_2$ in $\CDach$ contains a critical point of $f$.
\end{claim1}
\begin{proof}
Let $e_1,e_2\in E$ be two distinct edges of $\sT$ that join a critical fixed point $c\in C_f$ and a (repelling) fixed point $r\in R_f$. Suppose also that $e_1$ and $e_2$ correspond to the fixed internal rays of angles $\theta_1$ and $\theta_2$ in the immediate basin $\Omega_c$ of $c$. Let $U$ be a complementary component of the Jordan curve $e_1\cup e_2$. 
We argue by contradiction and assume that the statement is false,  that is, $U\cap C_f=\emptyset$.

Since $U$ is a Jordan domain and, by our assumption, $\overline{U}\cap f(C_f) = \overline{U} \cap C_f = \{c\}$, 
 it follows that each connected component $U'$ of $f^{-1}(U)$ is a Jordan domain and $f|\overline{U'}\colon \overline{U'} \to \overline{U}$ is a homeomorphism \cite[Proposition 2.8]{PilgrimJordan}. Since $f$ is injective in a neighborhood of $r$ and the edges $e_1$ and $e_2$ are fixed under $f$,  we conclude that $U$ is a connected component of $f^{-1}(U)$. However, this is not possible, because $f|\overline{U}$ is not injective. Indeed, one of the two internal rays of angles $\theta_1\pm\frac{1}{d_c}$ in $\Omega_c$ belongs to $\overline{U}$. At the same time, both of these rays and the edge $e_1$ are mapped by $f$ onto $e_1$, so $f|\overline{U}$ is not injective. This gives the
desired contradiction.
\end{proof}

\begin{claim2}
$\#E\geq 2\#F$.
\end{claim2}
\begin{proof}

Let $Q$ be an arbitrary face of the planar embedded graph $\sT$. Since $\sT$ is bipartite, each connected component of the boundary $\partial Q$ is traced by a circuit $(e_0, e_1, \dots, e_{n-1})$ of even length $n\geq 2$. Claim~1 implies that $Q$ cannot be bounded by just one circuit of length~$2$. Consequently, the total length of circuits that bound $Q$, which we simply denote by $\# \partial Q$, is at least~$4$.  Note also that each edge $e\in E$ appears exactly twice among the circuits tracing the boundary components of the faces of~$\sT$.  It follows that
\[2\#E=\sum_{Q\in F} \#\partial Q \geq 4 \#F,\]
which gives the desired inequality.
\end{proof}

Denote by $\kappa$ the number of connected components of the graph $\sT$.  Euler’s formula implies that
\begin{equation}\label{eq:Euler-Tisch} 
\#V-\#E+\#F=\kappa+1.
\end{equation}
Using Claim~2 and substituting \eqref{eq:fix-Tisch} and \eqref{eq:RH-Tisch} into \eqref{eq:Euler-Tisch}, we get
\begin{equation}\label{eq:tish-equation}
\kappa+1=\#V-\#E+\#F\leq \#V-\frac{1}{2}\#E \leq (d+1)-\frac{1}{2}(2d-2)=2.
\end{equation}
Consequently, $\kappa=1$, that is, the Tischler graph of $f$ is connected. This finishes the proof of Theorem \ref{thm:Tisch-connected}.\\

The argument above shows that all inequalities in \eqref{eq:tish-equation} are in fact equalities. This implies the following statement.

\begin{cor}\label{cor:Tisch-structure}
Let $f$ be a critically fixed rational map with $\deg(f)\geq 2$.  Then the
following statements are true:
\begin{enumerate}[label=\text{(\roman*)},font=\normalfont]
\item\label{cor:i} The vertex set of the Tischler graph $\Tisch(f)$ coincides with the set $\fix(f)$ of fixed points of~$f$.  In particular, each fixed point of $f$ lies on a fixed internal ray.
\item\label{cor:ii}  The boundary $\partial Q$ of each face $Q$ of $\Tisch(f)$ is connected and traced by a circuit of length $4$. That is, each face is either a quadrilateral or a bigon with a sticker inside.
\item\label{cor:iii}  The number of faces of $\Tisch(f)$ equals $\deg(f) - 1$.
\end{enumerate}
\end{cor}

\section{Combinatorial classification of critically fixed rational maps}\label{sec:classification}

The goal of this section is to prove Theorem \ref{thm:classif_cr_fixed}. Namely, we will construct a canonical bijection $$\Phi\colon \operatorname{ConPlanGr}_n\to \operatorname{CritFixRat}_{n+1}$$ between the set $\operatorname{ConPlanGr}_n$ of isomorphism classes of connected planar embedded graphs with exactly $n$ edges (and without loops) and the set $\operatorname{CritFixRat}_{n+1}$ of conformal conjugacy classes of critically fixed rational maps of degree $n+1$.  Theorem \ref{thm:classif_cr_fixed} will then immediately follow. To describe the map $\Phi$, we first need to introduce some auxiliary constructions.

\subsection{From graphs to maps}

Let $\sG=(V,E)$ be a connected planar embedded graph with exactly $n$ edges. 
We now define a (topological) branched covering map $\widetilde{f}_\sG\colon \CDach\to\CDach$ obtained from the identity map on $\CDach$ by \emph{blowing up} each edge of the graph $\sG$ in the sense of Pilgrim and Tan Lei \cite[\S 2.5]{PT}. 
Informally, we cut the sphere $\CDach$ open along the interior of each edge and glue in a closed Jordan region ``patch'' inside each slit along the boundary.  Each complementary component of the union of the patches is then mapped by $\widetilde{f}_\sG$ homeomorphically onto the corresponding face of $\sG$. At the same time, $\widetilde{f}_\sG$ maps the interior of each Jordan region patch homeomorphically onto the complement of the respective edge. We present a formal  construction below.

Suppose that $E=\{e_1,\dots,e_n\}$ and $F=\{U_1,\dots,U_m\}$ are the edge and face sets of the graph~$\sG$, respectively. 
For each edge $e_j\in E$, we choose a closed Jordan region $D_j\subset \CDach$ that satisfies the following properties:
\begin{enumerate}[font=\normalfont, label=(B\arabic*)]
\item\label{B1} The endpoints of $e_j$ lie on $\partial D_j$ and the interior of $e_j$ is in $D_j ^\circ$.  In particular,  the endpoints of $e_j$ split $\partial D_j$ into two Jordan arcs, which we denote by $e'_j$ and $e''_j$.
\item\label{B2} Distinct closed Jordan regions $D_{j_1}$ and $D_{j_2}$ intersect only at the common vertices (if any) of the edges $e_{j_1}$ and $e_{j_2}$, that is, $D_{j_1}\cap D_{j_2}= e_{j_1}\cap e_{j_2}$ for $j_1\neq j_2$.
\end{enumerate}

Note that the two properties above and \cite[Theorem A.6(ii)]{BuserGeometry} immediately imply that
\begin{enumerate}[font=\normalfont, label=(B\arabic*)]
\setcounter{enumi}{2} 
\item\label{B3} The arcs $e_j'$ and $e''_j$ are isotopic to $e_j$ relative to the vertex set $V$ for all $j=1,\dots,n$.
\end{enumerate} 
 
Set $E':=\bigcup_{j=1}^{n}\{e'_j,e''_j\}$. Clearly, $\sG':=(V,E')$ is a connected planar embedded graph, which we call a \emph{blow-up} of $\sG$.  By construction,  the interior $\interior{D_j}$ is a face of $\sG'$ for each $j=1,\dots, n$.  Moreover,  each face $U_k$, $k=1,\dots,m$, of $\sG$ contains a unique face of $\sG'$, which we denote by $U'_k$.  Note that then $$\CDach=\bigcup_{j=1}^n D_j \cup \bigcup_{k=1}^m U'_k.$$

\begin{example} A blow-up of the graph $\sG$ from Figure \ref{fig:Graph_G} is shown in Figure \ref{fig:Graph_Gprime}. The regions in dark gray color in Figure \ref{fig:Graph_Gprime} correspond to the Jordan region patches.
\end{example}

The map $\widetilde{f}_\sG\colon \CDach\to\CDach$ can now be described as follows.  First, we define $\widetilde{f}_\sG$ on the realization~$\GC'$ of $\sG'$ so that 
\begin{enumerate}[font=\normalfont, label=(B\arabic*)]
\setcounter{enumi}{3} 
\item\label{B4} $\widetilde{f}_\sG(v)=v$ for each $v\in V$;
\item\label{B5} $\widetilde{f}_\sG|e'_j\colon  e'_j\to e_j$ and $\widetilde{f}_\sG|e''_j\colon e''_j\to e_j$ are homeomorphisms for each $j=1,\dots,n$.
\end{enumerate}
Note that the image $\widetilde{f}_\sG(\GC')$ is the realization $\GC$ of $\sG$. Then, we continuously extend $\widetilde{f}_\sG$ to the whole sphere $\CDach$ so that
\begin{enumerate}[font=\normalfont, label=(B\arabic*)]
\setcounter{enumi}{5} 
\item\label{B6} $\widetilde{f}_\sG|U'_k:U'_k\to U_k$ is an orientation-preserving homeomorphism for all $k=1,\dots,m$; 
\item\label{B7}$\widetilde{f}_\sG|\interior{D_j}:\interior{D_j}\to\CDach\setminus e_j$ is an orientation-preserving homeomorphism for all $j=1,\dots,n$. 
\end{enumerate}

\begin{figure}
  \centering
  \begin{subfigure}[b]{0.45\textwidth}
    \begin{overpic}
      [width=7cm, tics=10,
      ]{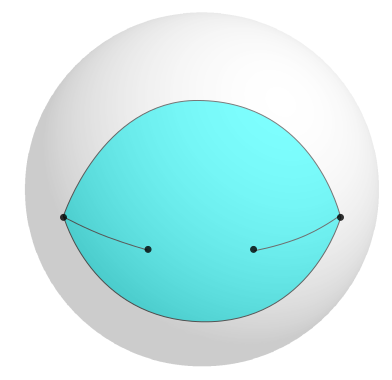}
      \put(87.5,37){$v_1$}
      \put(9.5,37){$v_{-1}$}     
      \put(39.5,32){$v_{-c}$}
      \put(59,32){$v_c$}   
      \put(108,45){$\longleftarrow$}
      \put(110,50){$\widetilde{f}_\sG$}
    \end{overpic}
  \end{subfigure}
  \hfill
  \begin{subfigure}[b]{0.45\textwidth}
    \begin{overpic}
      [width=7cm, tics=10,
      ]{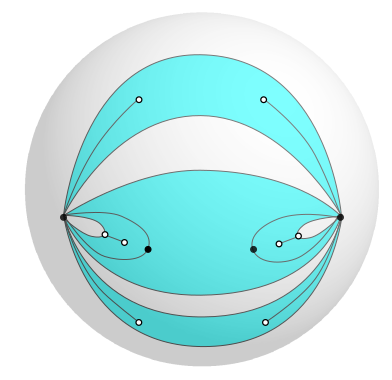}
      \put(87.5,37){$v_1$}
      \put(9.5,37){$v_{-1}$}     
      \put(39.5,32){$v_{-c}$}
      \put(59,32){$v_c$}             
    \end{overpic}
  \end{subfigure}
  
  \caption{The map $\widetilde{f}_\sG$ obtained by blowing up edges of the graph $\sG$ (on the left).}
\label{fig:Blow_map}
\end{figure}

It is straightforward to check that the continuous map $\widetilde{f}_\sG\colon \CDach\to\CDach$ obtained in this way is an orientation-preserving branched covering map (this easily follows from \cite[Corollary A.14]{THEBook}). Note that the topological degree of $\widetilde{f}_\sG$ equals $\# E + 1 = n+1$ and $\widetilde{f}_\sG$ is a local homeomorphism outside $V$. Furthermore, the local degree of $\widetilde{f}_\sG$ at a vertex $v\in V$ is given by $\deg_\sG(v)+1$, and thus the set $C_{\widetilde{f}_\sG}$ of critical points of $\widetilde{f}_\sG$ equals~$V$. Condition \ref{B4} implies that each critical point of $\widetilde{f}_\sG$ is fixed,  which means that $\widetilde{f}_\sG$ is postcritically-finite. 

\begin{example}
The map $\widetilde{f}_\sG$ obtained by blowing up edges of the graph $\sG$ from Figure \ref{fig:Graph_G} is illustrated in Figure \ref{fig:Blow_map}. The graph on the right corresponds to the preimage $\widetilde{f}_\sG^{-1}(\GC)$ of the realization $\GC$ of~$\sG$. The small white discs correspond to the points from $\widetilde{f}_\sG^{-1}(V)\setminus V$, where $V=\{v_{-1},v_{-c},v_c,v_1\}$ is the vertex set of~$\sG$. Each gray and blue face on the right is mapped by $\widetilde{f}_\sG$ homeomorphically onto the face of $\sG$ of the corresponding color on the left.
\end{example}

As follows from the construction, the map $\widetilde{f}_\sG$ is not uniquely defined. However, it is uniquely determined up to the following natural equivalence relation. 

We say that two postcritically-finite branched covering maps $f,g\colon \Sp\to\Sp$ on a topological $2$-sphere $\Sp$ are \emph{combinatorially equivalent} if they commute up to an isotopy relative to the postcritical set; that is,  there exist orientation-preserving homeomorphisms $h_0, h_1\colon \Sp\to\Sp$ that are isotopic relative to $P_f$ and satisfy $h_0\circ f = g\circ h_1$.  A celebrated theorem due to William Thurston characterizes those postcritically-finite branched covering maps on $\Sp$ (with hyperbolic orbifolds) that are combinatorially equivalent to a rational map on $\CDach$ \cite{DH_Th_char}.  Moreover,  it says that if two postcritically-finite rational maps with hyperbolic orbifolds are combinatorially equivalent then they are conformally conjugate. 

It follows from \cite[Proposition 2]{PT} that the combinatorial equivalence class of $\widetilde{f}_\sG$ depends only on the isomorphism class of $\sG$.  Since $\sG$ is connected, \cite[Corollary 3]{PT} implies that $\widetilde{f}_\sG$ is combinatorially equivalent to a critically fixed rational map $f_\sG\colon\CDach\to\CDach$ with $\deg(f_\sG)=n+1$.  By the rigidity portion of Thurston's theorem,  the map $f_\sG$ is uniquely defined up to conformal conjugacy.  

Thus we obtain a map $\Phi\colon \operatorname{ConPlanGr}_n\to \operatorname{CritFixRat}_{n+1}$ induced by the blow-up construction: $\Phi$ sends the isomorphism class of the connected planar embedded graph $\sG$ to the conformal conjugacy class of the rational map $f_\sG$.   The injectivity of the map $\Phi$ follows from \cite[Theorem 1.3]{Pilgrim_crit_fixed}.  To prove Theorem \ref{thm:classif_cr_fixed} it remains to show that $\Phi$ is surjective. We do this in the remainder of this section.

\subsection{From maps to graphs}

In the following, we assume that $f\colon\CDach\to\CDach$ is a critically fixed rational map with $\deg(f)\geq 2$. Also, let $\sT:=\Tisch(f)$ be the Tischler graph of $f$ with the vertex, edge, and face sets denoted by $V_\sT$, $E_\sT$, and $F_\sT$, respectively. As before, $R_f$ is the set of landing points of the fixed internal rays of $f$.  By Corollary \ref{cor:Tisch-structure},  $V_\sT=C_f\sqcup R_f=\fix(f)$ and $\#F_\sT = \deg(f)-1$.  Moreover,  each face of $\sT$ is simply connected and has exactly two critical points on the boundary.

For each face $Q$ of $\sT$, we choose a Jordan arc $e(Q)$ joining the two critical points of $f$ on $\partial Q$ so that $\inter(e(Q))\subset Q$.  Note that $e(Q)$ subdivides the face $Q$ into two open triangular domains with three fixed points on the boundary (one from $R_f$ and two from $C_f$).  Let $H_\sT$ be the set of all such domains, which we  call \emph{half-faces} of $\sT$.  In other words,  $H_\sT$ is the set of complementary components of the union 
$$\bigcup_{e\in E_\sT}e\cup\bigcup_{Q\in F_\sT} e(Q).$$ 

\begin{definition*}
The planar embedded graph with the vertex set $C_f$ and the edge set $\{e(Q): Q\in F_\sT\}$ is called a \emph{charge graph} of $f$ and denoted by $\Charge(f)$. 
\end{definition*}

\begin{example}
A charge graph of the map $f$ from Section \ref{sec:example} is shown in solid lines in Figure \ref{fig:Charge_f}.  Here,  the dashed arcs represent the edges of the Tischler graph of $f$ from Figure \ref{fig:Tisch_f}. Note that $\Charge(f)$ is isomorphic to the graph $\sG$ from Figure \ref{fig:Graph_G}.
\end{example}

We deduce now several properties of the charge graph.

\begin{lemma}\label{lem:faces-of-charge}
The following statements are true for the charge graph $\sG:=\Charge(f)$:
\begin{enumerate}[label=\text{(\roman*)},font=\normalfont]
\item\label{charge:i} The graph $\sG$ is connected. 
\item\label{charge:ii} Each face of $\sG$ contains exactly one point from $R_f$. 
\item\label{charge:iii} If $U$ is the face of $\sG$ containing a point $r\in R_f$, then
\begin{equation}\label{eq:face}
U=  \bigcup_{\substack{W\in H_\sT,\\ \text{s.t. $r\in \partial W$}}} \left(\overline{W}\setminus e(W)\right),
\end{equation}
where $e(W)$ denotes the unique edge of $\sG$ that belongs to the boundary $\partial W$ of a half-face~$W\in H_\sT$.
\end{enumerate}
\end{lemma}

\begin{proof} Let $U$ be a face of $\sG$ and $e(Q)$ be an edge of $\sG$ on $\partial U$.  By construction,  $U$ must contain the unique point $r\in R_f\cap \partial Q$. 

Suppose now that $e_0,e_1,\dots,e_{k-1}$ are the edges of $\sT$ incident to the vertex $r$ and labeled in cyclic order. Let $Q_0,Q_1,\dots, Q_{k-1}$ and $W_0,W_1,\dots, W_{k-1}$ be the faces and half-faces of $\sT$ surrounding the vertex $r$ 
and labeled so that $e_j, e_{j+1} \subset \partial Q_j$ and $W_j\subset Q_j$ for each $j=0,\dots,k-1$ (where $e_k=e_0$). Then every $W_j$ is an open Jordan domain with $\partial W_j= e_j\cup e_{j+1}\cup e(Q_j)$. 
It follows that $\left(e(Q_0),e(Q_1),\dots, e(Q_{k-1})\right)$ is a circuit in $\sG$ that traces the boundary of the face~$U$. In fact, 
$$U= \bigcup_{j=0}^{k-1}\left(\overline{W_j}\setminus e(Q_j)\right),$$
which completes the proof of parts \ref{charge:ii} and \ref{charge:iii}.

To prove \ref{charge:i},  note that $\sG$ has exactly $\#C_f$ vertices,  $\#F_\sT= \deg(f)-1$ edges,  and $\#R_f$ faces (the latter follows from \ref{charge:ii}).  Hence $\sG$ is connected by Euler's formula as $\#C_f+\# R_f=\# \fix(f)=\deg(f)+1$.  This finishes the proof of the lemma.
\end{proof}

The next proposition implies that $\Phi\colon \operatorname{ConPlanGr}_{n}\to \operatorname{CritFixRat}_{n+1}$ is surjective, which completes the proof of Theorem \ref{thm:classif_cr_fixed}.

\begin{prop}\label{prop:diagonal-blow}
Each critically fixed rational map $f$ with $\deg(f)\geq 2$ may be obtained from a charge graph of $f$ by blowing up its edges. That is,  there is a charge graph $\sG=\Charge(f)$ with the vertex set $V=C_f$ and edge set $E=\{e_1,\dots, e_n\}$, as well as closed Jordan regions $D_1,\dots, D_n$, 
that together satisfy conditions~{\normalfont \ref{B1}--\ref{B7}} with $\widetilde{f}_\sG= f$.
\end{prop}

\begin{remark} Since the charge graph of the map $f$ from Section \ref{sec:example} is isomorphic to the graph $\sG$ from Figure \ref{fig:Graph_G}, Proposition \ref{prop:diagonal-blow} implies that the map $f$ and the map $\widetilde{f}_\sG$ from Figure \ref{fig:Blow_map}, obtained by blowing up the edges of $\sG$, are combinatorially equivalent.
\end{remark}

\begin{proof}[Proof of Proposition \ref{prop:diagonal-blow}] As before,  let $\sT:=\Tisch(f)$ be the Tischler graph of the given critically fixed rational map $f$ with the vertex set $V_\sT=C_f\sqcup R_f$ and edge set $E_\sT$.  Suppose also that $F_\sT=\{Q_1,\dots, Q_n\}$ is the face set of $\sT$.  First, we select the edges $e(Q_1),\dots, e(Q_n)$ of a charge graph of $f$ as follows.

Let $Q_j$, $j=1,\dots, n$, be an arbitrary face of the Tischler graph $\sT$. By Corollary~\ref{cor:Tisch-structure}\ref{cor:ii}, $\partial Q_j$ is traced by a circuit 
$$c_{j,1},\, e_{j,1}, \, r_{j,1}, \, e_{j,2}, \, c_{j,2}, \, e_{j,3},\, r_{j,2},\, e_{j,4},\, c_{j,1},$$
where $e_{j,1}, \, e_{j,2}, \, e_{j,3}, \, e_{j,4}\in E_\sT$, $c_{j,1},\, c_{j,2} \in C_f$, and $r_{j,1},\, r_{j,2}\in R_f$.  
Since the edges $e_{j,1}, \, e_{j,2}, \, e_{j,3}, \, e_{j,4}$ are fixed under $f$ and $f$ is locally injective near $r_{j,1}$ and $r_{j,2}$, there are exactly two connected components $Q'_j$ and $Q''_j$ of $f^{-1}(Q_j)$ such that $e_{j,1}, \, e_{j,2} \subset \partial Q'_j$ and  $e_{j,3}, \, e_{j,4} \subset \partial Q''_j$.  Note that components $Q'_j$ and $Q''_j$ are disjoint and both belong to $Q_j$.  Moreover, the boundaries $\partial Q'_j$ and $\partial Q''_j$ are composed of internal rays at $c_{j,1}$ and $c_{j,2}$.  We may now choose the edge $e_j:=e(Q_j)$ so that $\inter(e_j) \subset Q_j\setminus (Q'_j \cup Q''_j)$; see Figure \ref{fig:Face}. The figure shows two possible configurations: when the face $Q_j$ is a quadrilateral (on the left), and when it is a bigon with a sticker inside (on the right).  In the figure, the components $Q'_j$ and $Q''_j$ of $f^{-1}(Q_j)$ are colored gray,  and the small white squares correspond to the preimages of $r_{j,1}$ and $r_{j,2}$ on $\partial Q'_j$ and $\partial Q''_j$.   The solid red arc connecting $c_{j,1}$ and $c_{j,2}$ represents the chosen edge $e_j$.

Now we construct the desired closed Jordan region $D_j$.  Set 
$$e'_j:=\overline{Q'_j\cap f^{-1}(e_j)} \text{ and } e''_j:=\overline{Q''_j\cap f^{-1}(e_j)}.$$
Since $Q_j\cap C_f = \emptyset$,  the maps $f|Q'_j\colon Q'_j\to Q_j$ and $f|Q''_j\colon Q''_j\to Q_j$ are homeomorphisms.  It follows that $e'_j$ and $e''_j$ are Jordan arcs joining $c_{j,1}$ and $c_{j,2}$.  In Figure \ref{fig:Face},  these arcs are represented by the dashed red curves.   
Let $D_j$ be the closed Jordan region bounded by the Jordan curve $e'_j \cup e''_j$ such that the interior $\interior{D_j}\subset Q_j$.

\begin{figure}
  \centering
  \begin{subfigure}[b]{0.45\textwidth}
    \begin{overpic}
      [width=7.0cm, tics=10, trim=0 0 0 50, clip,
      ]{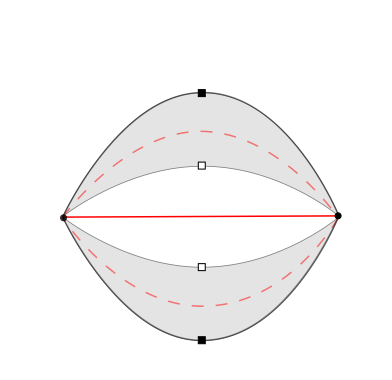}
      \put(87,37){$c_{j,2}$}
      \put(10,37){$c_{j,1}$}     
      \put(49,5){$r_{j,2}$}
      \put(40,13.5){$W''_{j,2}$}   
      \put(49,77){$r_{j,1}$}
      \put(40,66){$W'_{j,1}$}   
      \put(22,64){$e_{j,1}$}
      \put(74,64){$e_{j,2}$}       
      \put(22,18){$e_{j,4}$}
      \put(74,18){$e_{j,3}$}           
      \put(60,21){$e''_{j}$} 
      \put(60,59){$e'_{j}$}     
      \put(60,38){\textcolor{red}{$e_{j}$}}               
    \end{overpic}
  \end{subfigure}
  \hfill
  \begin{subfigure}[b]{0.45\textwidth}
    \begin{overpic}
      [width=7.0cm, tics=10, trim=0 0 0 50, clip,
      ]{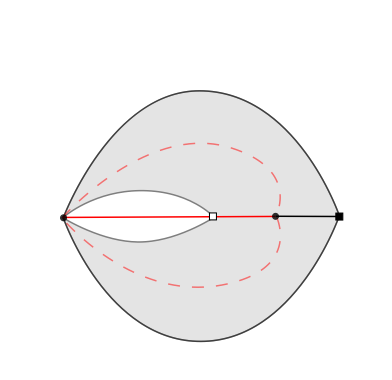}
      \put(87,37){$r_{j,2}$}
      \put(9,37){$c_{j,1}$}     
      \put(49,5){$e_{j,4}$}
      \put(40,13.5){$W''_{j,2}$}   
      \put(40,66){$W'_{j,1}$}         
      \put(49,77){$e_{j,1}$} 
      \put(74,38){$e_{j,3}$}   
      \put(74,44){$e_{j,2}$}  
      \put(87,45){$r_{j,1}$}                    
      \put(58,25){$e''_{j}$} 
      \put(58,57){$e'_{j}$}   
      \put(35,38){\textcolor{red}{$e_{j}$}}         
      \put(64,37){$c_{j,2}$}               
    \end{overpic}
  \end{subfigure}
  \caption{Constructing the edge $e_j=e(Q_j)$ and the closed Jordan region $D_j$ inside a face $Q_j$ of the Tischler graph when $Q_j$ is a quadrilateral (left) and a bigon with a sticker inside (right).}
\label{fig:Face}
\end{figure}

Suppose now that $\sG$ is the charge graph of $f$ with the vertex set $V=C_f$ and edge set $E=\{e_1,\dots, e_n\}$.  Let $F=\{U_1,\dots,U_m\}$ be the face set of $\sG$. Up to relabeling, by Lemma~\ref{lem:faces-of-charge}\ref{charge:ii} we may assume that $R_f=\{r_1,\dots, r_m\}$ and $r_k\in U_k$ for each $k=1,\dots,m$. 

We claim that the chosen charge graph $\sG$ and closed Jordan regions $D_1,\dots, D_n$ satisfy conditions \ref{B1}--\ref{B7} with $\widetilde{f}_\sG= f$, so that $f$ is obtained from $\sG$ by blowing up its edges.  Conditions \ref{B1}--\ref{B5} follow immediately from the construction.  Checking the remaining two conditions is also straightforward,  and we do this below for completeness.

Let $\sG':=(V,E')$ be the planar embedded graph with the vertex set $V=C_f$ and the edge set $E':=\bigcup_{j=1}^{n}\{e'_j,e''_j\}$. For each $j=1,\dots, n$, let $W_{j,1}$ and $W_{j,2}$ be the two half-faces of $\sT$ inside the face $Q_j$ labeled so that 
$$\text{$\partial W_{j,1}=e_{j,1}\cup e_{j,2}\cup e_j$ and  $\partial W_{j,2}=e_{j,3}\cup e_{j,4}\cup e_j$}.$$
If we set $W'_{j,1}:= (f|Q'_j)^{-1}(W_{j,1})$ and $W''_{j,2}:= (f|Q''_j)^{-1}(W_{j,2})$, then 
$$Q_j = W'_{j,1} \sqcup \inter(e'_j) \sqcup \interior{D_j} \sqcup \inter(e''_j) \sqcup W''_{j,2};$$
see Figure \ref{fig:Face}. For every $k=1,\dots,m$, let
\begin{equation}\label{eq:face-preim}
U'_k:=  \bigcup_{\substack{W\in H'_\sT,\\ \text{s.t. $r_k\in \partial W$}}} \left(\overline{W}\setminus e(W)\right),
\end{equation}
where $H'_T:=\bigcup_{j=1}^{n}\{W'_{j,1}, W''_{j,2}\}$ and $e(W)$ denotes the unique edge of $\sG'$ that belongs to the boundary of $W\in H'_T$.  An argument similar to the proof of Lemma \ref{lem:faces-of-charge} shows that $U'_k$ is the face of $\sG'$ containing $r_k$.  It follows that
the set $F'$ of faces of $\sG'$ is given by
$$F'=\{\interior{D_1},\dots, \interior{D_n}\}\sqcup \{U'_1,\dots, U'_m\}.$$
Since $f|W'_{j,1}\colon W'_{j,1}\to W_{j,1}$ and $f|W''_{j,2}\colon W''_{j,2}\to W_{j,2}$ are homeomorphisms for all $j=1,\dots, n$, and $f$ is locally injective near $r_k$ for each $k=1,\dots, m$,  equations \eqref{eq:face} and \eqref{eq:face-preim} imply that $f|U'_k\colon U'_k\to U_k$ is an orientation-preserving homeomorphism for every $k=1,\dots, m$. So condition \ref{B6} is satisfied.

Finally, we check \ref{B7} using a counting argument.  Let $p$ be an arbitrary point in the complement of the realization of $\sG$, say $p\in U_1$.  From the considerations above, all preimages of $p$ under $f$ lie in $U'_1 \cup \interior{D_1}\cup \dots\cup \interior{D_n}$.  In particular,  $\#(f^{-1}(p)\cap U_1')=1$.  At the same time, since $f$ sends $e'_j$ and $e''_j $ homeomorphically to $e_j$,  we have $\#(f^{-1}(p)\cap\interior{D_j})\geq1$ for all  $j=1,\dots, n$.  

By Corollary \ref{cor:Tisch-structure}\ref{cor:iii}, $n=\# F_\sT= \deg(f)-1$. Thus
$$n+1 = \deg(f)= \# f^{-1}(p) = 1 + \sum_{j=1}^{n}  \#(f^{-1}(p)\cap\interior{D_j}).$$
This implies that
 $\#(f^{-1}(p)\cap\interior{D_j})=1$ for all  $j=1,\dots, n$.  It follows that every $f|\interior{D_j}:\interior{D_j}\to \CDach\setminus e_j$ is an (orientation-preserving) homeomorphism,  which establishes condition \ref{B7}. 

We checked all the conditions necessary to conclude that $f$ is obtained from $\sG$ by blowing up its edges, so Proposition \ref{prop:diagonal-blow} is proven.
\end{proof}

\section{The global curve attractor problem}\label{sec:global-curve-attractor}

Here, we provide a positive answer to the global curve attractor problem for critically fixed rational maps, that is, prove Theorem \ref{thm:finite-curve-attractor}. First, we set up some notation.

Let $f\colon\CDach\to\CDach$ be a critically fixed rational map with $\deg(f)\geq 2$ and $\sG:=\Charge(f)$ be a charge graph of $f$ defined in the previous section. 
Suppose that $E=\{e_1,\dots,e_n\}$ is the edge set of $\sG$. By Proposition \ref{prop:diagonal-blow}, we may assume that $f$ is obtained from $\sG$ by blowing up its edges.

Suppose that $\mathscr{C}(f)$ is the set of all (unoriented) simple closed curves in $\CDach\setminus C_f=\CDach\setminus P_f$. Recall that a pullback of a curve $\gamma\in \mathscr{C}(f)$ under $f$ is a connected component of the preimage $f^{-1}(\gamma)$.  
We denote the isotopy class of a curve $\gamma\in \mathscr{C}(f)$ relative to $C_f$ by $[\gamma]$. 

Given a curve $\gamma\in\mathscr{C}(f)$ and a Jordan arc $e$ in $\CDach$ with endpoints in $C_f$, we define the (\emph{geometric}) \emph{intersection number between $\gamma$ and $e$}, denoted $\gamma \cdot e$, by
$$\gamma \cdot e := \min_{\gamma'\in [\gamma]} \#(\gamma' \cap e).$$

Similarly, we define the (\emph{geometric}) \emph{intersection number between $\gamma$ and the charge graph $\sG$}, denoted $\gamma \cdot \sG$, by
$$\gamma \cdot \sG := \min_{\gamma'\in [\gamma]} \#(\gamma' \cap \GC),$$
where $\GC$ is the realization of $\sG$.  The bigon criterion (see \cite[Sections 1.2.4 and 1.2.7]{FarbMargalit}) implies that 
$$\gamma \cdot \sG = \sum_{j=1}^n \gamma \cdot e_j.$$
We will refer to the intersection number $\gamma \cdot \sG$ as the \emph{complexity} of the curve $\gamma$ (relative to the charge graph $\sG$). 

The following key lemma relates the complexities of a curve $\gamma\in \mathscr{C}(f)$ and its pullbacks.

\begin{lemma}\label{lem:complexity-decreases}
Let $\gamma\in \mathscr{C}(f)$ be arbitrary and $\Delta(\gamma)$ be the set of all pullbacks of $\gamma$ under $f$. Then 
\begin{equation}\label{eq:contraction}
\sum_{\delta \in \Delta(\gamma)} \delta\cdot \sG \leq \gamma \cdot \sG ,
\end{equation}
that is, the total complexity of the pullbacks of $\gamma$ does not exceed the complexity of $\gamma$.

Moreover, if $\gamma \cdot e_j > 1$ for some edge $e_j\in E$, then 
\begin{equation}\label{eq:strict-contraction}
\sum_{\delta \in \Delta(\gamma)} \delta\cdot \sG < \gamma \cdot \sG.
\end{equation}
\end{lemma}

\begin{figure}
  \centering
  \begin{subfigure}[b]{0.45\textwidth}
    \begin{overpic}
      [width=7cm, tics=10,
      ]{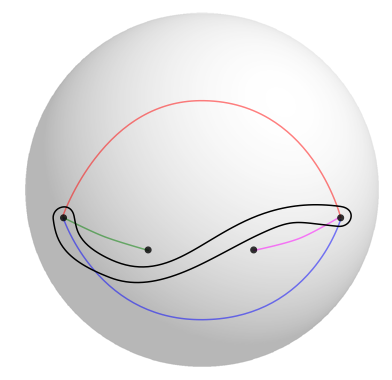}
      \put(50,10){$e_4$}
      \put(50,73){$e_3$}
      \put(26,38){$e_1$}
      \put(74,31){$e_2$}
      \put(108,45){$\longleftarrow$}
      \put(110,50){$\widetilde{f}_\sG$}
       \put(49,37){$\gamma$}
    \end{overpic}
\caption{}
    \label{fig:Curve}
  \end{subfigure}
  \hfill
  \begin{subfigure}[b]{0.45\textwidth}
    \begin{overpic}
      [width=7cm, tics=10,
      ]{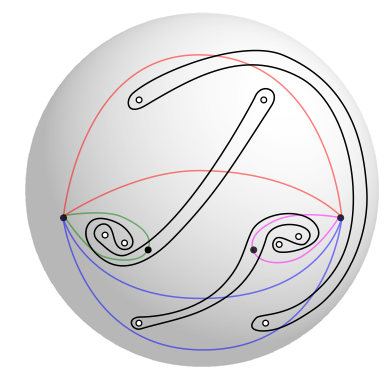}
      \put(80,80){$\delta_1$}
      \put(52,60){$\delta_2$}     
      \put(54,29){$\delta_3$}    
      \put(42,85){$e'_3$}
      \put(42,23.5){$e'_4$}
      \put(32,42){$e'_1$}
      \put(65,42){$e'_2$}
    \end{overpic}
   \caption{}
    \label{fig:Pullbacks}
  \end{subfigure}
  \caption{The pullbacks of a curve $\gamma$ under the map $\widetilde{f}_\sG$ from Figure~\ref{fig:Blow_map}.}
\label{fig:Curve_and_Pullbacks}
\end{figure}

\begin{example}
Instead of the map $f$ from Section \ref{sec:example}, let us consider the map $\widetilde{f}_\sG$ obtained by blowing up the edges of the graph $\sG$ from Figure \ref{fig:Graph_G}  (recall that $\widetilde{f}_\sG$ is combinatorially equivalent to~$f$).  The colored arcs in Figure \ref{fig:Curve_and_Pullbacks} correspond to the edges of $\sG$ and of its blow-up: 
the arcs colored green, purple, red, and blue in Figure~\ref{fig:Pullbacks} are mapped by $\widetilde{f}_\sG$ homeomorphically to the green, purple, red, and blue edges in Figure~\ref{fig:Curve}, respectively.  Also, the small white discs in Figure~\ref{fig:Pullbacks} indicate the non-critical preimages of the critical points of $\widetilde{f}_\sG$.  Let $\gamma$ be the simple closed curve shown in black in Figure~\ref{fig:Curve}. 
Note that $\gamma \cdot e_1 =  \gamma \cdot e_2 = 1$ and $\gamma \cdot e_3 = \gamma \cdot e_4 = 2$,  and thus 
$\gamma \cdot \sG = 6$.  The curve $\gamma$ has three pullbacks under  $\widetilde{f}_\sG$, denoted by $\delta_1,\, \delta_2, \, \delta_3$ and shown in black in Figure~\ref{fig:Pullbacks}. These pullbacks satisfy  $\delta_1 \cdot \sG =0 $ and $\delta_2 \cdot \sG=\delta_3 \cdot \sG=1$. Hence 
$$\delta_1 \cdot \sG + \delta_2 \cdot \sG + \delta_3 \cdot \sG < \gamma \cdot \sG,$$
which agrees with Lemma \ref{lem:complexity-decreases} (since $\gamma \cdot e_3 =\gamma \cdot e_4 =2$).
\end{example}

\begin{proof}[Proof of Lemma \ref{lem:complexity-decreases}]
Without loss of generality,  we may assume that the given curve $\gamma\in \mathscr{C}(f)$ 
is in ``minimal position'' with the realization $\GC$ (within its isotopy class~$[\gamma]$), that is, $$\gamma\cdot \sG = \# (\gamma \cap \GC)=\sum_{j=1}^n \#(\gamma \cap e_j). $$ 
Then $\gamma\cdot e_j = \#(\gamma \cap e_j)$ for every $e_j\in E$.\\

Let $D_j$, $e'_j$, and $e''_j$ for $j\in\{1,\dots,n\}$ be as in the proof of Proposition \ref{prop:diagonal-blow}. By \ref{B5}, $f$ sends each Jordan arc $e'_j$ homeomorphically to $e_j$. Hence $\#(\gamma \cap e_j)= \#(f^{-1}(\gamma) \cap e'_j)$. At the same time, by \ref{B3}, $e'_j$ is isotopic to $e_j$ relative to $C_f$,  and thus $\delta \cdot e_j = \delta \cdot e'_j$ for all pullbacks $\delta\in\Delta(\gamma)$. Consequently,
\begin{align*}
\sum_{\delta \in \Delta(\gamma)} \delta\cdot \sG &= \sum_{\delta \in \Delta(\gamma)}\,  \sum_{j=1}^n \delta\cdot e_j =  \sum_{j=1}^n\,  \sum_{\delta \in \Delta(\gamma)}  \delta\cdot e'_j  \\ \numberthis \label{eq:complexity_proof}
&  \leq \sum_{j=1}^n \, \sum_{\delta \in \Delta(\gamma)}  \#(\delta\cap e'_j)  =  \sum_{j=1}^n 
\#(f^{-1}(\gamma)\cap e'_j)=  \sum_{j=1}^n \#(\gamma \cap e_j) = \gamma\cdot \sG,
\end{align*}
which gives the desired inequality \eqref{eq:contraction}.\\

Suppose now that $\#(\gamma\cap e_j)= \gamma\cdot e_j > 1$ for some edge $e_j\in E$.   Let $P = \{p_1, p_2, \dots, p_m\}$, $m\geq2$, be the set of intersections between $\gamma$ and $e_j$ listed from one endpoint of $e_j$ to another.  Note that all these intersections have to be transverse, because $\gamma$ and $e_j$ are in minimal position. 
If we set $P':= e_j'\cap f^{-1}(\gamma)$ and $P'':= e_j''\cap f^{-1}(\gamma)$, we may write $P' = \{p'_1, p'_2, \dots, p'_m\}$ and $P'' = \{p''_1, p''_2, \dots, p''_m\}$ where $f(p'_k) = f(p''_k) = p_k$ for each $k = 1, \dots, m$.

By \ref{B7},  $f|D_j^\circ\colon D_j^\circ\to \CDach \setminus e_j$ is a homeomorphism.  Since $\gamma$ is a simple closed curve,  it follows that $f^{-1}(\gamma)\cap D_j$ consists of exactly $m$ pairwise disjoint Jordan arcs with endpoints in $P'\cup P''$.   We claim that one of these arcs must connect two points from $P'$ or two points from $P''$.  Indeed,  let us suppose that $p'_1\in P'$ and $p''_1\in P''$ lie on the components $\beta'$ and $\beta''$ of $f^{-1}(\gamma)\cap D_j$, respectively. Note that $\beta'\neq \beta''$; for otherwise $\gamma=f(\beta')$ and $\gamma$ has only one intersection with $e_j$.  Then $\beta' \cap \beta'' = \emptyset$,  which implies that either both endpoints of $\beta'$ are in $P'$ or both endpoints of $\beta''$ are in $P''$.  Indeed,  if $\beta'$ connects $p'_1\in P'$ to a point in $P''$, then $\beta'$ disconnects $\beta''\subset D_j$ from $P'$.  Hence $\beta''$ must connect $p''_1\in P''$ to another point in $P''$.  Without loss of generality, we assume in the following that both endpoints of $\beta'$ are in $P'$.

Let $\delta$ be the pullback of $\gamma$ with $\beta'\subset \delta$.  Then $\delta$ and $e'_j$ form a bigon and
$$\delta \cdot e_j = \delta \cdot e'_j \leq \#(\delta \cap e'_j) -2.$$
It follows that the inequality in \eqref{eq:complexity_proof} is strict, which completes the proof of \eqref{eq:strict-contraction}.
\end{proof}

Let $\mathcal{A}(f)$ 
be the set of isotopy classes of simple closed curves that intersect each edge of $\sG$ at most once, that is, 
$$\mathcal{A}(f)=\{[\gamma]: \gamma\in \mathscr{C}(f) \text{ with  $\gamma \cdot e_j \leq 1$ for all $e_j\in E$}\}.$$ 
Note that $\mathcal{A}(f)$ is finite because $\sG$ is connected (the order in which $\gamma$ crosses the edges of $\sG$ uniquely determines the isotopy class of $\gamma$).

Theorem \ref{thm:finite-curve-attractor} is an immediate consequence of the following statement.

\begin{prop}\label{prop:attractor}
For every curve $\gamma\in\mathscr{C}(f)$, each pullback $\delta$ of $\gamma$ under $f^n$ satisfies $[\delta]\in\mathcal{A}(f)$ for all sufficiently large $n$.
\end{prop}
\begin{proof}
Note that if $[\gamma]\in \mathcal{A}(f)$, then $[\delta]\in\mathcal{A}(f)$ for every pullback $\delta$ of $\gamma$.  Indeed,  arguing as in the proof of Lemma \ref{lem:complexity-decreases}, we have:
$$\delta \cdot e_j = \delta \cdot e_j'\leq  \#(\delta \cap e'_j) \leq \#(f^{-1}(\gamma) \cap e'_j) = \#(\gamma\cap e_j) = \gamma\cdot e_j \leq 1$$
for all $e_j\in E$.  For an arbitrary $\gamma\in\mathscr{C}(f)$,  the statement now follows from Lemma \ref{lem:complexity-decreases} by induction on the complexity $\gamma\cdot \sG$.
\end{proof}

\begin{remark} In fact, a slightly stronger statement is true.  It is straightforward to check that for each curve $\gamma$ with ${[\gamma]\in\mathcal{A}(f)}$ there exists a pullback $\delta$ of $\gamma$ under $f$ with $[\delta]=[\gamma]$; see \cite[Lemma 5.5]{HluPro} for an argument. Consequently,  $\mathcal{A}(f)$ is the global curve attractor of $f$.
\end{remark}

\section{Acknowledgments}
The author gratefully acknowledges the support of the Studienstiftung des Deutschen
Volkes.  He was also partially supported by an AMS-Simons Travel Grant.  The author would like to express sincere gratitude to his PhD advisers,
Dierk Schleicher and Daniel Meyer, for their guidance and encouragement during
his graduate studies.  He is grateful to Mario Bonk, Dzmitry Dudko, Russell Lodge, Kevin Pilgrim,  Nikolai Prochorov, and Palina Salanevich for multiple insightful discussions,  and especially thankful to Kevin Pilgrim for introducing him to the problem of connectivity of Tischler graphs.  The author also thanks Indiana University Bloomington and the University of California, Los Angeles for their hospitality during research visits,  where a part of this work was carried out.

\bibliographystyle{alpha}
\bibliography{main}

\end{document}